\documentclass[12pt]{article}
\usepackage{amsmath,amsthm,amsfonts,amssymb}
\usepackage{hyperref}

\newtheorem{theorem}{Theorem}
\newtheorem{lemma}[theorem]{Lemma}

\theoremstyle{definition}
\newtheorem{definition}[theorem]{Definition}

\newcommand{\scrB}{\mathcal{B}}
\newcommand{\scrP}{\mathcal{P}}
\newcommand{\FF}{\mathbb{F}}
\newcommand{\PP}{\mathbb{P}}

\title{New Strongly Regular Graphs from Finite Geometries via Switching}
\author{Ferdinand Ihringer\thanks{Department of Mathematics: Analysis, Logic and Discrete Mathematics, Ghent University, Belgium, \href{mailto:ferdinand.ihringer@ugent.be}{ferdinand.ihringer@ugent.be} . The author is supported by a postdoctoral fellowship of the Research Foundation --- Flanders (FWO).},
Akihiro Munemasa\thanks{Graduate School of Information Sciences, Tohoku University, Sendai, Japan,
\href{mailto:munemasa@math.is.tohoku.ac.jp}{munemasa@math.is.tohoku.ac.jp} .}
}

\begin{document}

\maketitle

\begin{abstract}
  We show that the strongly regular graph on non-isotropic points of one type of the polar
  spaces of type $U(n, 2)$, $O(n, 3)$, $O(n, 5)$, $O^+(n, 3)$, and $O^-(n, 3)$
  are not determined by its parameters for $n \geq 6$.
  We prove this by using a variation of Godsil-McKay switching
  recently described by Wang, Qiu, and Hu. This also results
  in a new, shorter proof of a previous result of the first author which
  showed that the collinearity graph of a polar space is not determined by
  its spectrum. The same switching gives a linear algebra explanation
  for the construction of a large number of non-isomorphic designs.
\end{abstract}

\section{Introduction}

Two graphs $\Gamma$ and $\overline{\Gamma}$ are cospectral if their adjacency matrices $A$ and $\overline{A}$ are cospectral,
that is there exists an orthogonal matrix $Q$ with $Q^T A Q = \overline{A}$.
In 1982 Godsil and McKay described a possible choice for $Q$ which 
has an easy combinatorial description \cite{Godsil1982} -- nowadays known as Godsil--McKay switching --
and proved to be very useful in constructing cospectral graphs \cite{Abiad2018,Cioaba2018,Dam2005,Dam2006,Haemers2010,Kubota2017,Munemasa2015}.
Godsil--McKay switching can be described as follows. Here we denote the 
neighborhood of a vertex $x$ in a graph $\Gamma$ by $\Gamma(x)$.

\begin{theorem}[GM switching]\label{thm:gm_gen_t}
  Let $\Gamma$ be a graph whose vertex set is partitioned as $C_1 \cup \dots \cup C_t \cup D$.
Assume that $C:=C_1 \cup \dots \cup C_t$ is an equitable partition of the induced subgraph on $C$,
and that $|\Gamma(x) \cap C_i|/|C_i| \in \{ 0, \frac{1}{2}, 1 \}$ for all $x \in D$ and all $i \in \{1,\ldots, t \}$.
Construct a graph $\overline{\Gamma}$ from $\Gamma$
by modifying edges between $C_i$ and $D$ as follows:
\[\overline{\Gamma}(x)\cap C_i=\begin{cases}
C_i \setminus \Gamma(x) &\text{if $|\Gamma(x)\cap C_i|/|C_i|=\frac{1}{2}$,}\\
\Gamma(x)\cap C_i&\text{otherwise,}
\end{cases}\]
for $x\in D$.
Then $\overline{\Gamma}$ is cospectral with $\Gamma$.
\end{theorem}

Abiad and Haemers used Godsil--McKay switching \cite{Godsil1982} in \cite{Abiad2016} to obtain graphs
cospectral to the collinearity graphs of the symplectic polar space of rank at least $3$
over $\FF_2$. Barwick et al.\ generalized this, also using Godsil--McKay switching, 
to quadrics of rank at least $3$ over $\FF_2$ \cite{Barwick2016}.
Using different techniques, Kantor showed in \cite{Kantor1982a} that if a polar space has a spread, then one can define a
new (not necessarily non-isomorphic) strongly regular graph with the parameters as a finite classical
polar space $\PP$ over $\FF_q$ if $\PP$ possesses a spread of maximals.
For instance this implies nowadays that the strongly regular collinearity graphs of symplectic polar spaces
are not uniquely defined by their spectrum for sufficiently large $q$.
In \cite{Ihringer2016b} the first author generalized all of the results obtained by switching to all finite classical 
polar spaces of rank at least $3$ over $\FF_q$ by using purely geometrical arguments. 

This note provides an algebraic explanation for the construction given in \cite{Ihringer2016b}
by providing a specific family of orthogonal matrices $Q$. 
This gives rise to the following new switching which was recently discovered -- in a slightly more general form -- by Wang, Qiu, and Hu \cite{Wang2019}.
We discovered the same switching independently while looking for a linear algebra argument for the results
by the first author in \cite{Ihringer2016b}.

\begin{theorem}[WQH switching]\label{thm:2}
 Let $\Gamma$ be a graph whose vertex set is partitioned as $C_1 \cup C_2 \cup D$. 
 Assume that $|C_1|=|C_2|$ and that the induced subgraphs on $C_1$, $C_2$, and $C_1 \cup C_2$
 are regular, where the degrees in the induced subgraphs on $C_1$ and $C_2$ are the same.
Suppose that all $x \in D$ satisfy one of the following:
\begin{enumerate}
 \item $|\Gamma(x) \cap C_1| = |\Gamma(x) \cap C_2|$, or
 \item $\Gamma(x) \cap (C_1 \cup C_2) \in \{C_1, C_2\}$.
\end{enumerate}
Construct a graph $\overline{\Gamma}$ from $\Gamma$ by modifying the
edges between $C_1 \cup C_2$ and $D$ as follows:
\begin{align*}
 \overline{\Gamma}(x) \cap (C_1 \cup C_2) = 
 \begin{cases}
      C_1 & \text{ if } \Gamma(x) \cap (C_1 \cup C_2) = C_2,\\
      C_2 & \text{ if } \Gamma(x) \cap (C_1 \cup C_2) = C_1,\\
      \Gamma(x) \cap (C_1 \cup C_2) & \text{ otherwise,}
 \end{cases}
\end{align*}
for $x \in D$. Then $\overline{\Gamma}$ is cospectral with $\Gamma$.
\end{theorem}

Notice that we simplified the statement from \cite{Wang2019} where the stated conditions are also necessary.
In our case, the conditions are just sufficient.
One can show that if $t=1$ and $|C_1|=4$ in Theorem \ref{thm:gm_gen_t}, then it is equivalent to
$|C_1| = |C_2| = 2$ in Theorem \ref{thm:2}.

In the later sections we show that WQH switching indeed yields strongly regular graphs 
with the same parameters as the collinearity graphs of polar spaces, thus reproducing
the results from \cite{Ihringer2016b} with a much shorter proof.
We start by applying WQH switching to graphs coming from designs to reproduce a classical construction
for non-isomorphic $2$-designs due to Jungnickel and Tonchev \cite{Jungnickel1984,Jungnickel2010}
(see also Shrikhande \cite{Shrikhande1953}), using the linear algebra method.
Furthermore, we use the switching to show that all the strongly regular collinearity graphs coming from one type of non-isotropic points
in the ambient space of the polar spaces $U(n, 2)$, $O(n, 3)$, $O(n, 5)$, $O^+(n, 3)$, and $O^-(n, 3)$
are not uniquely determined by their parameters for $n \geq 6$.
This extends work by Hui and Rodrigues who showed that the strongly regular collinearity graph 
on non-isotropic points of $O^+(n, 2)$ and $O^-(n, 2)$ are not uniquely determined 
by their parameters for $n \geq 6$ \cite{Hui2016}.

\begin{definition}
 If we can apply Theorem~\ref{thm:2} to a graph $\Gamma$ for a pair $\{ C_1, C_2 \}$,
 then we call $\{ C_1, C_2 \}$ a \textit{switching set} of $\Gamma$.
 We refer to $\overline{\Gamma}$ (as in Theorem~\ref{thm:2}) as the graph \textit{obtained by switching}.
\end{definition}

\section{Block Designs and Grassmann Graphs}

A $2$-$(v, b, \lambda)$ design $(\scrP, \scrB)$ consists of a family $\scrB$ (blocks) of $b$-sets of an
$v$-element set $\scrP$ (points) such that every pair of points lies in exactly $\lambda$ blocks.
All $k$-spaces of an $n$-dimensional vector space over $\FF_q^n$ form a $2$-design. Jungnickel \cite{Jungnickel1984}
constructed many non-isomorphic $2$-designs with the same parameters.
We reproduce a special case of his construction by applying WQH switching.
For $\lambda=1$, the \textit{block graph} of $(\scrP, \scrB)$ is the graph with $\scrB$ as vertices where two vertices are adjacent if
the blocks intersect.

  Let $(\scrP, \scrB)$ be a $2$-$(v, b, 1)$ design with a $2$-$(v', b, 1)$ subdesign $(\scrP', \scrB')$.
  Let $p_1, p_2 \in \scrP'$. Let $C_1$ be the set of all blocks of $\scrB'$ containing $p_1$ and not $p_2$,
  and let $C_2$ be the set of all blocks of $\scrB'$ containing $p_2$ and not $p_1$.
  
\begin{theorem}\label{thm:switching_design}
  The pair $\{ C_1, C_2 \}$ is a switching set of the block graph $\Gamma$ of $(\scrP, \scrB)$.
\end{theorem}
\begin{proof}
  The points $p_1$ and $p_2$ lie in exactly one block together in $\scrB'$, so 
  $|C_1| = \frac{v'-1}{b-1} - 1 = |C_2|$.
  
  First we show that $C_1 \cup C_2$ is an equitable partition. Each block in $C_i$ is adjacent to
  all blocks in $C_i$. Each block in $C_1$ is adjacent to $b-1$ blocks in $C_2$.
  As $|C_1| = |C_2|$, this suffices.
  
  We have to verify that every block $B$ not in $C_1 \cup C_2$ satisfies $|\Gamma(B) \cap C_1| = |\Gamma(B) \cap C_2|$
  or $\Gamma(B) \cap (C_1 \cup C_2) \in \{ C_1, C_2 \}$.
  The block $B$ through $p_1$ and $p_2$ clearly has $\Gamma(B) \cap (C_1 \cup C_2) = C_1 \cup C_2$.
  A block $B \in \scrB'$, which does not contain $p_1$ or $p_2$, meets $b$ elements of $C_i$ for $i \in \{1, 2 \}$ as $\lambda=1$.
  A block in $\scrB\setminus\scrB'$ contains $0$ or $1$ point of $\scrP'$ as $\lambda=1$.
  A block $B \in \scrB \setminus \scrB'$ which contains exactly one $p_i$ for $i \in \{1, 2\}$, satisfies $\Gamma(B) \cap (C_1 \cup C_2) = C_i$.
  A block $B \in \scrB$ which does not contain $p_1$ or $p_2$ meets $0$ elements of $C_i$ if $B \cap \scrP' = \varnothing$ and $1$ element of $C_i$
  if $|B \cap \scrP| = 1$ for all $i \in \{ 1, 2\}$.
\end{proof}

The graph $\overline{\Gamma}$ obtained by switching with respect to the
switching set given in Theorem~\ref{thm:switching_design} is also a block graph.
Indeed, it is the block graph
of the modified design obtained from $(\scrP,\scrB)$ by replacing $C_i$ by
\[\{L\cup\{p_{i'}\}\setminus\{p_{i}\}\mid L\in C_i\},\]
where $\{i,i'\}=\{1,2\}$.
This is
a special case of the construction given in \cite{Jungnickel1984}. 

As an example 
let $\scrP$ be the set of all $1$-spaces of $\FF_q^n$ and let $\scrB$ be the set of all $2$-spaces of $\FF_q^n$.
For $(\scrP', \scrB')$ we choose an $s$-space $S$ of $\FF_q^n$ with $2 < s < n$.

\begin{lemma}
  The graphs obtained from the block graph $\Gamma$ of the above design $(\scrP, \scrB)$ by switching are non-isomorphic to $\Gamma$.
\end{lemma}
\begin{proof}
  Let $\ell$ be the $2$-space spanned by $p_1$ and $p_2$.
  Let $T$ be a $3$-space which meets $S$ in a line $\ell' \neq \ell$ with $p_1 \in \ell'$.
  After switching, all blocks in $\{ B \in \scrB \setminus \scrB': p_1 \in B\}$ form a maximal clique of size $q^2+q$.
  It is well-known that all maximal cliques of $\Gamma$ have size $q^2+q+1$ (cf. \cite[Proposition 2.20]{Mussche2009}).
\end{proof}

There are many similar constructions for designs in the literature. 
For instance, we can replace the $2$-space by a $k$-space with $1 < k < n$.
Our emphasis here is that we
provide an explanation for this construction by considering the block graph and using the linear algebra method.

\section{Polar Spaces}\label{sec:appl}

We apply Theorem~\ref{thm:2} to several strongly regular graphs related to finite classical polar spaces,
so that we obtain new, non-isomorphic strongly regular graphs with the same parameters.
We refer the reader to \cite[Section 9.4]{Brouwer1989}, \cite{Hirschfeld1991}, \cite{Ihringer2016b} and \cite[Chapter 8]{Tits1974}
for more detailed descriptions of finite classical polar spaces. In the following we give a brief overview.
We use projective notation, so we denote $1$-spaces of $\FF_q^n$ as \textit{points}, $2$-spaces as \textit{lines}, and $(n-1)$-spaces
as \textit{hyperplanes}.
There are six finite classical polar spaces: $O^+(2d, q)$, $U(2d, \sqrt{q})$, $U(2d+1, \sqrt{q})$, $Sp(2d, q)$, $O(2d+1, q)$, $O^-(2d+2, q)$.
All polar spaces can be defined by a polarity $\perp$ of $\FF_q^n$, except for $q$ even in the case of the quadrics $O^+(2d, q)$, $O(2d+1, q)$ and $O^-(2d+2, q)$. As we are not showing any results 
for quadrics in even characteristic, we assume in the following that either $q$ is odd or that the polar space
is not a quadric.
A polarity is a incidence preserving bijection between $k$-spaces and $(n-k)$-spaces of order two.
Call a subspace $S$ \textit{isotropic} if $S \subseteq S^\perp$. 
A polar space over $\FF_q^n$ consists of
all isotropic subspaces. We call a subspace $T$ of $\FF_q^n$ \textit{degenerate} if $T^\perp \cap T$ is non-empty.
If $T$ is degenerate, then $T^\perp/(T \cap T^\perp)$ is a non-degenerate finite classical polar space.

Recall the following property of polar spaces:
  Let $\{ x, y, z \}$ be non-isotropic points of the ambient space of a finite classical polar space 
  such that $x, y \in z^\perp$ and $x \in y^\perp$.
  Then $\langle x, y, z \rangle$ is isomorphic to $U(3,\sqrt{q})$ or $O(3, q)$.

Throughout the following let $\PP$ denote a polar space of rank $d \geq 3$ over a finite field of order $q$ 
with an associated polarity $\perp$ such as each
isotropic $(d-1)$-space lies in $q^e+1$ maximal isotropic subspaces. 

\section{Collinearity Graphs of Polar Spaces}

In this section we recover a special case of the result of \cite{Ihringer2016b} using Theorem~\ref{thm:2}.
The result in \cite{Ihringer2016b} was the original motivation for the research presented here.
The collinearity graph $\Gamma$ of $\PP$ has the isotropic points of $\PP$ as vertices and two points 
are adjacent if they are collinear.
The graph $\Gamma$ is well-known to be strongly regular.
Let $P$ be an isotropic $m$-space, where $2 \leq m \leq d$.
Let $L_1$ and $L_2$ be hyperplanes of $P$ with $L_1 \neq L_2$.
Set 
\begin{align*}
  &C_1 = L_1 \setminus L_2, && C_2 = L_2 \setminus L_1.
\end{align*}

\begin{theorem}\label{thm:polar_coll}
  The pair $\{C_1, C_2\}$ is a switching set of $\Gamma$.
  If $d=m \geq 3$, then the graph $\overline{\Gamma}$ obtained by switching is non-isomorphic to $\Gamma$. 
\end{theorem}
\begin{proof}
 First we show that $\{ C_1, C_2 \}$ is a switching set. 
 Clearly, the induced subgraph
 on $C_1 \cup C_2$ is a complete graph, hence we only have to investigate the size of
 $x^\perp \cap C_1$ and $x^\perp \cap C_2$ for $x \notin C_1 \cup C_2$.
 If $x \in P^\perp$, then $x^\perp \cap (C_1 \cup C_2) = C_1 \cup C_2$.
 For $x \notin P^\perp$, let $L = x^\perp \cap P$.
 If $L = L_i$ for $i \in \{ 1,2 \}$, then $x^\perp \cap (C_1 \cup C_2) = C_i$.
 Otherwise, $\dim(x^\perp \cap L_1) = \dim(x^\perp \cap L_2)$, so $|x^\perp \cap C_1| = |x^\perp \cap C_2|$.
 Hence, $\{ C_1, C_2 \}$
 is a switching set.
 
 It is easy to see that the graph $\Gamma_0$ considered in Section 4 of \cite{Ihringer2016b} corresponds
 to $\Gamma$ (and $\Gamma_1$ to $\overline{\Gamma}$). Non-isomorphy is shown there
 and it requires $d \geq 3$.
\end{proof}

Recall that $\Gamma$ is a strongly regular graph, so $\overline{\Gamma}$
is a strongly regular graph too.
We do not know if all graphs constructed in \cite{Ihringer2016b} can be reconstructed using
Theorem~\ref{thm:2}.

\section{Non-isotropic Points of Unitary Polar Spaces}\label{sec:polarity}

Let $\perp$ be a polarity on $\FF_q^n$ such that there are non-isotropic points,
so the associated polar space $\PP$ is one of $O^+(2d, q)$, $U(2d, \sqrt{q})$, $U(2d+1, \sqrt{q})$, $O(2d+1, q)$, $O^-(2d+2, q)$.
In the case of the quadrics, $q$ is odd. Let $\Gamma$ be the graph with the non-isotropic points of $\PP$ as vertices
and points $x, y$ are adjacent if $x \in y^\perp$. For this paper we call $\Gamma$ a \textit{polarity graph}.\footnote{This is slightly non-standard as the polarity graph usually includes isotropic points as well.}
As a motivation for Section \ref{sec:quads}, we show that $\Gamma$ is not determined by their spectrum for $n \geq 4$.
Note that $\Gamma$ is in general not strongly regular, but has several nice properties. 
For instance it is $K_{n+2}$-free and therefore for many parameters,
it gives the best known constructive lower bound on certain Ramsey numbers \cite{Alon1997}.
For $q=4$, in the case of $U(n, 2)$, the graph $\Gamma$ is indeed strongly regular such that the complement $\overline{\Gamma}$ of $\Gamma$ has the parameters (with $\epsilon = (-1)^n$):
\begin{align*}
  &v = 2^{n-1}(2^n-\epsilon)/3, && k = (2^{n-1}+\epsilon)(2^{n-2}-\epsilon),\\
  &\lambda = 3 \cdot 2^{2n-5} - \epsilon 2^{n-2}-2, &&\mu = 3 \cdot 2^{n-3} (2^{n-2}-\epsilon).
\end{align*}
One can see as follows that $\Gamma$ is strongly regular (as we are not aware of a reference): In \cite[C14]{Hubaut1975} a strongly regular graph $\Gamma'$ on the non-isotropic points of $U(n, 2)$ is defined, where two points are adjacent if they span a degenerate subspace. As a non-degenerate lines isomorphic to $U(2, 2)$ contains exactly two non-isotropic points,
which are pairwise orthogonal, $\Gamma' = \overline{\Gamma}$.

Let $3 \leq m \leq d$.
Let $L_1$ and $L_2$ be non-isotropic subspaces of dimension $m-1$ 
such that $L_1 \cap L_2$ is the radical of $L_1$ and $L_2$ and has dimension $m-2$,
that is $L_1 \cap L_2 = (L_1 \cap L_1^\perp) \cap (L_2 \cap L_2^\perp)$.
Set $P = \langle L_1, L_2 \rangle$, $p = L_1 \cap L_2$,
$C_1 = L_1 \setminus L_2$, and $C_2 = L_2 \setminus L_1$.

\begin{theorem}\label{thm:herm_ni_coll}
  The pair $\{C_1, C_2\}$ is a switching set of $\Gamma$.
\end{theorem}
\begin{proof}
  Let $x$ be a non-isotropic point in $L_i$ for some $i \in \{1,2\}$.
  As $x \notin x^\perp$, $x^\perp$ meets $P$ in a hyperplane.
  Clearly, $p \subseteq x^\perp$ by construction.
  As $x \in L_i$ is non-isotropic, $x^\perp \cap L_i$ contains no non-isotropic points.
  For $\{ i,j \} = \{ 1, 2\}$, $x^\perp \cap L_j$ contains all or no 
  non-isotropic points of $L_j$, depending on
  whether $L_j = L_i^\perp \cap P$ or not. Thus, the induced subgraph
  on $C_1 \cup C_2$ is a complete bipartite graph or an empty graph.
  Hence, the regularity condition of Theorem~\ref{thm:2} is satisfied.
  
  Next we have to consider $x \notin C_1 \cup C_2$.
  For $x$ a non-isotropic point not in $L_1 \cup L_2$, we have
  that either $P \subseteq x^\perp$ or $x^\perp \cap P$ is a hyperplane of $P$.
  If $P \subseteq x^\perp$, then clearly $|x^\perp \cap C_1| = |x^\perp \cap C_2|$.
  If $P \cap x^\perp$ is a hyperplane of $P$ and $p \nsubseteq P \cap x^\perp$,
  then $|x^\perp \cap L_1| = |x^\perp \cap L_2|$.
  If $P \cap x^\perp$ is a hyperplane of $P$ and $p \subseteq P \cap x^\perp$, then 
  $P \cap x^\perp$ is in $\{ L_1, L_2 \}$, so $x^\perp \cap (C_1 \cup C_2) \in \{ C_1, C_2 \}$,
  or $x^\perp \cap L_1 = p = x^\perp \cap L_2$, so $x^\perp \cap (C_1 \cup C_2)$ is empty.
  
  Hence, $\{ C_1, C_2 \}$ is a switching set.
\end{proof}

We show non-isomorphy for $U(n, \sqrt{q})$ for $m=3$. We believe that the same is true for $m>3$, but do
not intend to show it.
The argument for $O^+(n, q)$, $O^-(n, q)$, and $O(n, q)$
is analogous, but slightly more tedious as the perp of $U(3, \sqrt{q})$ in $U(n, \sqrt{q})$ is always 
isomorphic to $U(n-3, \sqrt{q})$, while the perp of $O(3, q)$ is either $O^+(n-3, q)$, $O^-(n-3, q)$,
or $O(n-3, q)$.

\begin{lemma}\label{cor:herm_ni_triple_ad}
  For a clique $\{ x, y, z \}$ of $\Gamma$ the value $|\Gamma(x) \cap \Gamma(y) \cap \Gamma(z)|$
  is independent of our choice of $\{ x, y, z \}$.
\end{lemma}
\begin{proof}
  Since $\Gamma(x) \cap \Gamma(y) \cap \Gamma(z)$ is contained in
  $\langle x, y, z \rangle^\perp \cong U(n-3, \sqrt{q})$, $|\Gamma(x) \cap \Gamma(y) \cap \Gamma(z)|$
  is the number of non-isotropic points in $U(n-3, \sqrt{q})$ which is independent of our choice of $\{ x, y, z \}$.
\end{proof}

\begin{theorem}\label{thm:herm_ni_ni}
  Suppose that $m=3$ and $P/p$ isomorphic to $U(2, \sqrt{q})$.
  If $n \geq 6$, then the graph $\overline{\Gamma}$ by switching is non-isomorphic to $\Gamma$. 
\end{theorem}
\begin{proof}
  Let $x, y, z$ be pairwise adjacent vertices not in $P$ with $L_x = x^\perp \cap P$, $L_y = y^\perp \cap P$,
  $L_z = z^\perp \cap P$ such that $L_x = L_1$, $L_x$, $L_y$ and $L_z$ are non-isotropic lines in $P$ 
  with $L_y \neq L_z$ and $L_y \cap L_x = L_z \cap L_x \neq p$.
  
  \textbf{Claim.} There exist $x$, $y$, and $z$ as described.
  
  Assume for now that the claim is true. Then $\Gamma(x) \cap \Gamma(y) \cap \Gamma(z)$ contains the unique non-isotropic point of
  $L_1 \cap L_y \cap L_z$, while $L_2 \cap L_y \cap L_z$ contains no non-isotropic points, so
  \begin{align*}
   |\overline{\Gamma}(x) \cap \overline{\Gamma}(y) \cap \overline{\Gamma}(z)| = |\Gamma(x) \cap \Gamma(y) \cap \Gamma(z)| - 1
  \end{align*}
  as no other adjacencies change. By Corollary~\ref{cor:herm_ni_triple_ad}, $\overline{\Gamma}$ is non-isomorphic to $\Gamma$.
  
  We still have to show the claim.
  
  Take a $U(6, \sqrt{q})$ space containing $P$ with the decomposition
  \begin{align*}
   \langle p_0, q_0 \rangle \perp \langle p_1, q_1 \rangle \perp \langle p_2, q_2 \rangle,
  \end{align*}
  where the $p_i$ are isotropic points and $\langle p_i, q_i \rangle$ are non-degenerate lines.
  Let $\{ r_i, s_i \}$ be an orthogonal basis of $\langle p_i, q_i \rangle$, that is in particular $s_i \in r_i^\perp$.
  Without loss of generality $P = \langle p_0, r_1, r_2 \rangle$,
  where $p = p_0$ and $L_1 = \langle p_0, r_1 \rangle$.
  Set 
  \begin{align*}
   &x = s_1 + p_2, && y = r_0, && z = s_0 + p_2.
  \end{align*}
  We now have the desired properties:
  The line $\langle r_1, r_2 \rangle$ is isomorphic to $U(2, \sqrt{q})$ and $p$ is
  orthogonal to $\langle r_1, r_2 \rangle$, so $P/p$ is isomorphic to $U(2, \sqrt{q})$.
  The points $x$, $y$, and $z$ are pairwise orthogonal, so $\{ x,y,z \}$ is a clique.
  We have that $x^\perp \cap P = \langle p_0, r_1 \rangle$.
  Furthermore, $L_y = y^\perp \cap P = \langle r_1, r_2 \rangle$, 
  but $r_2 \notin z^\perp \cap P = L_z$, so $L_y \neq L_z$.
  Additionally, $L_y \cap L_x = \langle r_1 \rangle = L_z \cap L_x$, so 
  in particular $p$ is not in $L_y$ or $L_z$.
\end{proof}

We conclude that there are graphs cospectral, but non-isomorphic, to $\Gamma$ for $n \geq 6$.
In particular, for $q=4$ we obtain new strongly regular graphs with parameters as above.

\section{Non-isotropic Points of Quadrics}\label{sec:quads}

Let $\PP$ be a polar space of type $O(2d+1, q)$, $O^+(2d, q)$, or $O^-(2d, q)$,
where $q\geq 3$ odd and $d \geq 2$, induced by some quadratic form $\sigma$ on $\FF_q^{n}$.
Here $n=2d+1$ if $\PP \cong O(2d+1, q)$ and $n=2d$ otherwise.
For a non-isotropic point $p$ we have that $\sigma(p)$ is a square 
or a non-square of $\FF_q$.
For $O(2d+1, q)$, we may choose the quadratic form $\sigma$ in
such a way that this is equivalent to $p^\perp$ isomorphic 
to $O^+(2d,q)$, respectively, $O^-(2d,q)$.
Set $V_{+} = \{ p: \sigma(p) \text{ square}\}$
and $V_{-} = \{ p: \sigma(p) \text{ non-square}\}$.
Let $\Gamma_\epsilon$ be the graph with the points in $V_\epsilon$ 
as vertices where two points $x, y$ are adjacent if $x \in y^\perp$.
It is well known that this is a strongly regular graph with parameters
\begin{align*}
  &v = 3^d(3^d+\epsilon)/2, && k = 3^{d-1}(3^d-\epsilon)/2,\\
  &\lambda = \mu = 3^{d-1} (3^{d-1}-\epsilon)/2,
\end{align*}
if $\PP \cong O(2d+1, 3)$ \cite[\S8.10]{Hubaut1975}, with parameters
\begin{align*}
  &v = 3^{d-1}(3^d-\zeta)/2, && k = 3^{d-1}(3^{d-1}-\zeta)/2,\\
  &\lambda = 3^{d-2} (3^{d-1}+\zeta)/2, && \mu = 3^{d-1} (3^{d-2}-\zeta)/2,
\end{align*}
if $\PP \cong O^\zeta(2d, 3)$ \cite[p. 377]{Hubaut1975},\footnote{Notice that $\mu$ is given incorrectly in \cite{Hubaut1975}. One can find the correct parameters on \url{https://www.win.tue.nl/~aeb/graphs/srghub.html}.} and with parameters
\begin{align*}
  &v = 5^d(5^d+\epsilon)/2, && k = 5^{d-1}(5^d-\epsilon)/2,\\
  &\lambda = 5^{d-1} (5^{d-1}+\epsilon)/2, && \mu = 5^{d-1} (5^{d-1}-\epsilon)/2,
\end{align*}
if $\PP \cong O(2d+1, 5)$ \cite[\S7.D]{Brouwer1984}.
To our knowledge it is not known if these strongly regular graphs
are determined by their parameters except for a finite number 
of small cases. 

Let $3 \leq m \leq d$.
Following the construction
of the previous section, 
let $L_1$ and $L_2$ be non-isotropic subspaces of dimension $m-1$ 
such that $L_1 \cap L_2$ is the radical of $L_1$ and $L_2$ and has dimension $m-2$,
that is $L_1 \cap L_2 = (L_1 \cap L_1^\perp) \cap (L_2 \cap L_2^\perp)$.
Set $P = \langle L_1, L_2 \rangle$, $p = L_1 \cap L_2$,
$C_1 = L_1 \setminus L_2$, and $C_2 = L_2 \setminus L_1$.
Furthermore, suppose that the points of $C_1$ and $C_2$ are of the same type,
so both are in $V_{\epsilon}$. Such $L_1$ and $L_2$ exist unless $P/p = O^+(2, 3)$.

The arguments in this section are almost identical to the ones for 
non-isotropic points in $U(n, \sqrt{q})$,
so we only elaborate on the parts of the argument which are slightly
different.

\begin{theorem}\label{thm:para_ni_q_eq_3}
  Let $\delta = -$ if $q \equiv 3 \pmod{4}$ and $\delta = +$ if $q \equiv 1 \pmod{4}$.
  The pair $\{C_1, C_2\}$ is a switching set of $\Gamma_\epsilon$.
  Suppose that $m=3$ and $P/p$ is isomorphic to $O^\delta(2, q)$.
  If $n \geq 7$,
  then the graph $\overline{\Gamma}_\epsilon$ obtained by switching is non-isomorphic to $\Gamma_\epsilon$.
\end{theorem}
\begin{proof}
  Recall that all non-isotropic points of $L_i$ are in $V_\epsilon$.
  Then the argument for $\{ C_1, C_2 \}$ being a switching set is identical to the proof of Theorem~\ref{thm:herm_ni_coll}.
  
  Write $\Gamma = \Gamma_\epsilon$ for brevity.
  In $\Gamma$ the number $|\Gamma(x) \cap \Gamma(y) \cap \Gamma(z)|$
  is independent of the choice of $x, y, z$ for a clique $\{ x, y, z \}$ of size $3$. 
  As in the proof of Theorem~\ref{thm:herm_ni_ni}, it suffices to show that $|\Gamma(x) \cap \Gamma(y) \cap \Gamma(z)|$
  is one larger than $|\overline{\Gamma}(x) \cap \overline{\Gamma}(y) \cap \overline{\Gamma}(z)|$
  for at least one choice of $\{ x, y, z \}$. A sufficient condition for such $\{ x, y, z \}$ to satisfy this conclusion is the one stated in the beginning of the proof of Theorem~\ref{thm:herm_ni_ni}.
  
    All that is left to do is to show that $\{ x, y, z \}$ exists.
  It is well known that 
a point of $O^\delta(2, q) \cap V_\epsilon$
is orthogonal to a point of $O^\delta(2, q) \cap V_\epsilon$.

 Write $P=\langle p\rangle\perp\langle r_1,r_2\rangle$,
$\langle r_1,r_2\rangle\cong O^\delta(2,q)$,
$L_1=\langle p,r_1\rangle$, $r_1\in r_2^\perp$.
Since $p^\perp/p$ contains an $O(5,q)$ space which contains
$O^\delta(4,q)=O^\delta(2,q)\perp O^+(2,q)$ space,
there exist $s_1\in V_\epsilon$ and a line $L\ni r_2$ with
$L\cong O^+(2,q)$ such that
\[p^\perp\supseteq\langle p\rangle \perp\langle r_1,s_1\rangle\perp L.\]
Since $n\geq7$, there exists an $O(7,q)$ space
\[ Q\perp\langle r_1,s_1\rangle\perp L\text{ with }
p\in Q\cong O(3,q).\]
Let $r_0, s_0 \in V_\epsilon \cap Q$ such that $\langle r_0, s_0, p_0 \rangle$ is a basis of $Q$ with $s_0 \in r_0^\perp$ and $p \notin r_0^\perp, s_0^\perp$.
Let $p_2$ be an isotropic point of $L$, and set
  \begin{align*}
   &x = s_1 + p_2, && y = r_0, && z = s_0 + p_2.
  \end{align*}
  We have that $x^\perp \cap P = \langle p_0, r_1 \rangle$.
  Furthermore, $L_y = y^\perp \cap P = \langle r_1, r_2 \rangle$, 
  but $r_2 \notin z^\perp \cap P = L_z$, so $L_y \neq L_z$.
  Additionally, $L_y \cap L_x = \langle r_1 \rangle = L_z \cap L_x$, so 
  in particular $p$ is not in $L_y$ or $L_z$.
\end{proof}

We want to emphasize that the construction and the non-isomorphy result are independent of $q$,
but that in the general case $\Gamma_\epsilon$ and $\overline{\Gamma}_\epsilon$ are not strongly
regular. 
For $O^-(6, 3)$, the general non-isomorphy result of Theorem 10 does not apply.
However, we have verified by computer that the above construction actually
yields a strongly regular graph not isomorphic to $\Gamma_-$.
For $O^+(6, 3)$, the graph remains unchanged by the switching, but it is non-isomorphic to the block graph
of $\text{AG}(3, 3)$, another strongly regular graph with the same parameters.

\section{Future Work}

There are infinite families of strongly regular graphs on the non-isotropic points of
$U(n, \sqrt{q})$ \cite[C14]{Hubaut1975} and one type of non-isotropic points of $O(2d+1, q)$ \cite[\S7C]{Brouwer1984},
where adjacency is not defined by orthogonality, but lying on a tangent. We wonder if these
strongly regular graphs are determined by their spectrum or if an argument similar to the 
arguments in this paper can construct new graphs of the same type. The smallest case of $U(n, \sqrt{q})$, 
that is $q=4$, was covered in Section \ref{sec:polarity}.

In a celebrated result Keevash did not only show that designs exist, but that many designs exist \cite{Keevash2014}.
We are wondering if the probablility of the condition in Theorem \ref{thm:switching_design} is bounded away from $0$
for these designs.

\bibliographystyle{plain}


\end{document}